\documentclass[11pt,twoside,letterpaper]{article} 
\usepackage{times,fancyhdr}   
\usepackage{amsfonts}                        
\usepackage{amsthm}                      
\usepackage{amsmath}                   
\usepackage{amssymb}  
\usepackage[normalem]{ulem}
\usepackage{color}

\numberwithin{equation}{section}


\makeatletter 
\def\cleardoublepage{\clearpage\if@twoside \ifodd\c@page\else%
    \hbox{}%
    \thispagestyle{empty}%
    \newpage%
    \if@twocolumn\hbox{}\newpage\fi\fi\fi} 
\makeatother 

\def \C{\mathbb{C}}

\def \N{\mathbb{N}}

\def \Q{\mathbb{Q}}

\def \Z{\mathbb{Z}}

\def \a{\alpha}
\def \b{\beta}
 
\def \sn{\mbox{sn}^2(z)}

\def \dn{\mbox{dn}^2(z)}  
    
\def \ns{\mbox{ns}^2(z)}
\def \nss{\mbox{ns}^2(2z)}
\def \nc{\mbox{nc}^2(z)}
\def \nd{\mbox{nd}^2(z)}  

\def \sno{\mbox{sn}^2z}
\def \ssno{\mbox{sn}^4z}
\def \cno{\mbox{cn}^2z} 
\def \dno{\mbox{dn}^2z}  
\def \ddno{\mbox{dn}^4z}

\newtheorem{conjecture}{Conjecture}[section] 
\newtheorem{thm}{Theorem}[section]

\newtheorem{lem}{Lemma}[section]
\newtheorem{rem}{Remark}[section]

\newtheorem{example*}{Example}[section]

\begin{document}
\title{
{\begin{flushleft}
\vskip 0.45in
{\normalsize\bfseries\textit{ }}
\end{flushleft}
\vskip 0.45in
\bfseries\scshape On linear relations for Dirichlet series 
formed by recursive sequences of second order}}

\thispagestyle{fancy}
\fancyhead{}
\fancyhead[L]{In: Book Title \\ 
Editor: Editor Name, pp. {\thepage-\pageref{lastpage-01}}} 
\fancyhead[R]{ISBN 0000000000  \\
\copyright~2007 Nova Science Publishers, Inc.}
\fancyfoot{}
\renewcommand{\headrulewidth}{0pt}

\author{\bfseries\itshape Carsten Elsner\thanks{Fachhochschule f{\"u}r die Wirtschaft, 
University of Applied Sciences, Freundallee 15,
D-30173 Hannover, Germany \newline e-mail: carsten.elsner@fhdw.de},
Niclas Technau\thanks{Technische Universit{\"a}t Graz, Institut f\"ur Analysis 
und Zahlentheorie, Steyrergasse 30/II, A-8010 Graz, Austria 
\newline e-mail: technau@math.tugraz.at; or: 
University of York, Department of Mathematics, Heslington, York, YO10 5DD, UK,
\newline e-mail: niclas.technau@york.ac.uk}}

\date{}
\maketitle
\thispagestyle{empty}
\setcounter{page}{1}
\[\]
\[\]
\begin{abstract}
Let $F_n$ and $L_n$ be the Fibonacci and Lucas numbers, respectively. 
Four corresponding zeta functions in $s$ are defined by
\[\zeta_F(s) \,:=\, \sum_{n=1}^{\infty} \frac{1}{F_n^s}\,,\quad 
\zeta_F^*(s) \,:=\,\sum_{n=1}^{\infty} \frac{{(-1)}^{n+1}}{F_n^s}\,,\quad 
\zeta_L(s) \,:=\, \sum_{n=1}^{\infty} \frac{1}{L_n^s}\,,\quad 
\zeta_L^*(s) \,:=\, \sum_{n=1}^{\infty} \frac{{(-1)}^{n+1}}{L_n^s} \,.\]
For positive integers $s$ the transcendence of these values is known 
as well as algebraic independence or dependence results for sets of these 
numbers. In this paper, we investigate linear forms in the above zeta functions 
and determine the dimension of linear spaces spanned by such linear forms.
In particular, it is established that for any positive integer $m$ the solutions of
\[\sum_{s=1}^m \big( \,t_s\zeta_F(2s) + u_s\zeta_F^*(2s) 
+ v_s\zeta_L(2s) + w_s\zeta_L^*(2s) \,\big) \,=\, 0 \]
with $t_s,u_s,v_s,w_s \in {\Q}$ $(1\leq s\leq m)$ form a ${\Q}$-vector 
space of dimension $m$. This proves a conjecture from the Ph.D. thesis of M.\,Stein,
who, in 2012, was inspired by the relation $-2\zeta_F(2)+\zeta_F^*(2)+5\zeta_L^*(2)=0$.
All the results are also true for zeta functions in $s$ 
where the Fibonacci and Lucas numbers are replaced by numbers 
from sequences satisfying a second order recurrence formula. 
\end{abstract}
 
\vspace{.08in} \noindent \textbf{Keywords:} Fibonacci and Lucas numbers, 
linear independence, elliptic functions, q-series 


\pagestyle{fancy}  
\fancyhead{}
\fancyhead[EC]{Carsten Elsner and Niclas Technau}
\fancyhead[EL,OR]{\thepage}
\fancyhead[OC]{On linear relations for Dirichlet series 
formed by recursive sequences of second order}
\fancyfoot{}
\renewcommand\headrulewidth{0.5pt} 

\newpage
\section{Introduction, and statement of the results} \label{Sec1}
There are many results on series formed by powers of such integers which satisfy a linear recurrence relation of order two. By studying the algebraic
character of sets of such series, surprising identities were found. For instance, in the case of Fibonacci numbers $F_n$, let
\[x \,:=\, \sum_{n=1}^{\infty} \frac{1}{F_n^2}\,,\quad y \,:=\, \sum_{n=1}^{\infty} \frac{1}{F_n^4}\,,\quad z \,:=\, \sum_{n=1}^{\infty} \frac{1}{F_n^6} \,.\]
Then, we have the following {\em non-linear relation\/} between four series (cf. \cite{Elsner2}), namely
\begin{eqnarray*}
\sum_{n=1}^{\infty} \frac{1}{F_n^8} &=& \frac{15y}{14} + \frac{1}{378{(4x+5)}^2}\big( \,256x^6 - 3456x^5 + 2880x^4 + 1792x^3z - 11100x^3 + 20160x^2z \\
&& -\,10125x^2 + 7560xz + 3136z^2 - 1050z \,\big) \,.
\end{eqnarray*}

Now, in order to look at the situation from a more conceptional point of view,
let $\a,\b$ be algebraic numbers with $|\b|<1$ and $\a \b =-1$. We define
\[U_n \,:=\, \frac{\a^n - \b^n}{\a - \b}\,
,\quad V_n \,:=\, \a^n + \b^n \]
for $n\geq 0$. Both sequences, $U_n$ and $V_n$, 
satisfy the second order recurrence formula
\[X_{n+2} \,=\, (\a + \b)X_{n+1}+X_n \qquad (n\geq 0)\,.\] 
For $\b =(1-\sqrt{5})/2$ we get the Fibonacci numbers 
$U_n=F_n$ and the Lucas numbers $V_n=L_n=F_{n-1}+F_{n+1}$. 
Moreover, for positive integers $s$ we introduce 
the series
\[\begin{array}{lcllcl}
\Phi_{2s} &:=& \displaystyle {(\a - \b)}^{-2s}
\sum_{n=1}^{\infty} \frac{1}{U_n^{2s}} \,,
\qquad \quad & \Phi^*_{2s} &:=& 
\displaystyle {(\a - \b)}^{-2s}\sum_{n=1}^{\infty} 
\frac{{(-1)}^{n+1}}{U_n^{2s}} \,, \\ \\
\Psi_{2s} &:=& \displaystyle \sum_{n=1}^{\infty} 
\frac{1}{V_n^{2s}} \,,
\qquad \quad & \Psi^*_{2s} &:=& \displaystyle  
\sum_{n=1}^{\infty} \frac{{(-1)}^{n+1}}{V_n^{2s}} \,.
\end{array} \]
In this paper we focus on {\em linear forms with rational coefficients\/} in 
\[\big\{ \Phi_{2s},\Phi^*_{2s},\Psi_{2s},\Psi^*_{2s} \,:\, s=1,\dots,m \big\} \]
and prove a conjecture of M.\,Stein on the dimension of the kernel of such forms;
see (\ref{eq: defining equation of the vectorspace in question}) and
Conjecture\,\ref{conj: Stein} at the end of this introductory section.
 
For a survey on irrationality, transcendence 
and algebraic independence results for series 
involving reciprocal Fibonacci and Lucas numbers,
we refer the reader to Sections\,1.1 to 1.3 in \cite{Stein}. 
Here, we present an outline devoted to the most 
important interim results which finally lead to the problem treated in this paper.
At the beginning in 1989, Andr{\'e}-Jeannin \cite{Jeannin} 
proved the irrationality of the series
\[\sum_{n=1}^{\infty} \frac{1}{F_n}\,,\quad 
\sum_{n=1}^{\infty} \frac{{(-1)}^n}{F_n}\,,\quad 
\sum_{n=1}^{\infty} \frac{1}{L_n}\,,\quad 
\sum_{n=1}^{\infty} \frac{{(-1)}^n}{L_n} \,,\]
where the underlying idea was inspired by Ap{\'e}ry's proof 
of the irrationality of $\zeta(3)$. Eight years later,
D.\,Duverney, Ke.\,Nishioka, Ku.\,Nishioka, and
I.\,Shiokawa \cite{Duverney} succeeded in proving 
the transcendence of the numbers
\begin{equation}
\sum_{n=1}^{\infty} \frac{1}{F_n^{2s}}\,,\quad 
\sum_{n=1}^{\infty} \frac{1}{L_n^{2s}}\,,\quad 
\sum_{n=1}^{\infty} \frac{1}{F_{2n-1}^s}\,,\quad 
\sum_{n=1}^{\infty} \frac{1}{L_{2n}^s} 
\label{EQ1.10}
\end{equation}
for any positive integer $s$. These results are derived 
using Yu.\,Nesterenko's theorem on 
Ramanujan functions \cite{Yuri} as follows. 
Let $K$ and $E$ denote the complete elliptic integrals of the first 
and second kind, respectively, with the modulus 
$k\in {\C}\setminus \{ 0,\pm 1\}$, defined by
\begin{equation}
K\,=\,K(k)\,=\,\int_0^1 \frac{dt}{\sqrt{(1-t^2)(1-k^2t^2)}} \,,\qquad 
E\,=\,E(k)\,=\,\int_0^1 \sqrt{\frac{1-k^2t^2}{1-t^2}}\,dt \,.
\label{EQ1.20}
\end{equation}
Their relationship to the well-known Ramanujan functions 
\begin{eqnarray*}
P(z) &=& 1 - 24\sum_{n=1}^{\infty} \sigma_1(n)z^n \,,\\
Q(z) &=& 1 + 240\sum_{n=1}^{\infty} \sigma_3(n)z^n \,,\\
R(z) &=& 1 - 504\sum_{n=1}^{\infty} \sigma_5(n)z^n 
\end{eqnarray*}
with
\[\sigma_r(n) \,=\, \sum_{d|n} d^r \qquad (r=1,2,\dots) \]
are given by
\begin{eqnarray*}
P(q^2) &=& {\Big( \,\frac{2K}{\pi}\,\Big)}^2\Big( \,
\frac{3E}{K}-2+k^2\,\Big) \,,\\
Q(q^2) &=& {\Big( \,\frac{2K}{\pi}\,\Big)}^4
\big( \,1-k^2+k^4\,\big) \,,\\
R(q^2) &=& {\Big( \,\frac{2K}{\pi}\,\Big)}^6 
\frac{1}{2}\big( \,1+k^2\,\big) \big( \,1-2k^2\,\big) 
\big( \,2-k^2\,\big)\,.
\end{eqnarray*}
By Nesterenko's theorem, for any algebraic number $q$,
the quantities $P(q^2)$, $Q(q^2)$, and $R(q^2)$ 
are algebraically independent over ${\Q}$, and so are the
quantities $K/\pi$, $E/\pi$ , and $k$. \\
Now, the series from (\ref{EQ1.10}) can be written 
as series of hyperbolic functions. Applying some identities from 
I.J.\,Zucker \cite{Zucker}, one can express the 
latter series in terms of so-called $q$-series defined by
\[\begin{array}{lcllcl}
A_{2j+1}(q) &=& \displaystyle \sum_{n=1}^{\infty} 
\frac{n^{2j+1}q^{2n}}{1-q^{2n}} \,,\qquad 
B_{2j+1}(q) &=& \displaystyle \sum_{n=1}^{\infty} 
\frac{{(-1)}^{n+1} n^{2j+1}q^{2n}}{1-q^{2n}} \,, \\ \\
C_{2j+1}(q) &=& \displaystyle \sum_{n=1}^{\infty} 
\frac{n^{2j+1}q^{n}}{1-q^{2n}} \,,\qquad 
D_{2j+1}(q) &=& \displaystyle \sum_{n=1}^{\infty} 
\frac{{(-1)}^{n+1} n^{2j+1}q^{n}}{1-q^{2n}} \,, 
\end{array} \]
where $q$ is a real algebraic number. 
Finally, these $q$-series can be expressed 
as polynomials in $K/\pi$, $E/\pi$ , and $k$, 
so that finally  the transcendence of 
the series from (\ref{EQ1.10}) follows 
from the algebraic independence of $K/\pi$, $E/\pi$,
and $k$ over ${\Q}$. For more details, see Section\,1.4 in \cite{Stein}.
\\
\\
This was the state when the first-named author 
of this paper started his joint work with Sh.\,Shimomura 
and I.\,Shiokawa on this subject. They extended it to
problems on algebraic independence and dependence over ${\Q}$. 
At the beginning, these authors only considered the Fibonacci 
and Lucas zeta functions, defined by 
\[\zeta_F(s) \,:=\, \sum_{n=1}^{\infty} \frac{1}{F_n^s}\,,\quad 
\zeta_F^*(s) \,:=\,\sum_{n=1}^{\infty} \frac{{(-1)}^{n+1}}{F_n^s}\,,\quad 
\zeta_L(s) \,:=\, \sum_{n=1}^{\infty} \frac{1}{L_n^s}\,,\quad 
\zeta_L^*(s) \,:=\, \sum_{n=1}^{\infty} \frac{{(-1)}^{n+1}}{L_n^s} \,.\]
It is shown in \cite{Elsner2} that the three numbers in each of the sets
\[\begin{array}{ll}
\big\{ \,\zeta_F(2),\zeta_F(4),\zeta_F(6)\,\big\} \,,\qquad 
& \big\{ \,\zeta_F^*(2),\zeta_F^*(4),\zeta_F^*(6)\,\big\} \,,\\ \\
\big\{ \,\zeta_L(2),\zeta_L(4),\zeta_L(6)\,\big\} \,,\qquad 
& \big\{ \,\zeta_L^*(2),\zeta_L^*(4),\zeta_L^*(6)\,\big\}
\end{array} \]
are algebraically independent over ${\Q}$, 
and that for any integer $s\geq 4$ each of the series 
$\zeta_F(2s),\zeta_F^*(2s),\zeta_L(2s)$ and $\zeta_L^*(2s)$ can be
expressed as rational functions in the three series 
of the same type for $s=2,4,6$, i.e., for $s\geq 4$ we have
\[\begin{array}{ll}
\zeta_F(2s) \,\in \, {\Q}\big( \,\zeta_F(2),\zeta_F(4),\zeta_F(6)\,\big) \,,\qquad 
& \zeta_F^*(2s) \,\in \, {\Q}\big( \,\zeta_F^*(2),\zeta_F^*(4),\zeta_F^*(6)\,\big)
\,, \\ \\
\zeta_L(2s) \,\in \, {\Q} ( \,\zeta_L(2),\zeta_L(4),\zeta_L(6)\, )
 \,,\qquad 
& \zeta_L^*(2s) \,\in \, {\Q} (\zeta_L^*(2),\zeta_L^*(4),\zeta_L^*(6)) \,.
\end{array} \]
For an explicit example we refer to the identity for $\zeta_F(8)$ given at the beginning of this section. 
A few years later the authors describe in \cite{Elsner4} all subsets of
\[\Gamma \,:=\, \big\{ \,\zeta_F(2),\zeta_F(4),\zeta_F(6),\zeta_F^*(2),\zeta_F^*(4),\zeta_F^*(6),
\zeta_L(2),\zeta_L(4),\zeta_L(6),\zeta_L^*(2),\zeta_L^*(4),\zeta_L^*(6)\,\big\} \]
with either algebraically dependent or algebraically independent numbers. 
They prove that every four numbers in $\Gamma$ are algebraically dependent over $\Q$, 
whereas every two distinct numbers are algebraically independent over ${\Q}$. Moreover, there are 198 of the 220 three-element subsets of $\Gamma$, each containing
algebraically independent numbers over ${\Q}$. For the remaining 22 three-element subsets 
of $\Gamma$, explicit algebraic relations are given. A complete list of
these relations is contained in the appendix of \cite{Stein}. 
With regard to the main result of this paper, we cite the only linear relation among them,
\begin{equation}
-2\Phi_2 + \Phi_2^* + \Psi_2^* \,=\, 0 \,.
\label{EQ1.30}
\end{equation}
Note that 
\[\begin{array}{lcccrcl}
\zeta_F(2s) &=& 5^s\Phi_{2s} \,,\qquad & \zeta_F^*(2s) &=& 5^s\Phi_{2s}^* \,,\\
\zeta_L(2s) &=& \Psi_{2s} \,,\qquad & \zeta_L^*(2s) &=& \Psi_{2s}^* \,.
\end{array} \]
A more general result for the Fibonacci zeta function at positive even 
integers is presented in \cite{Elsner5}. The authors establish 
an algebraic independence criterion based on the nonvanishing 
of the Jacobian determinant of a quadratic system of polynomials,
which is the kernel of the proof that for positive integers
$s_1<s_2<s_3$ the values $\zeta_F(2s_1)$, $\zeta_F(2s_2)$, 
and $\zeta_F(2s_3)$ are algebraically independent over ${\Q}$ 
if and only if at least one of the numbers $s_i$ is even. \\

The main result in M.\,Stein's Ph.D. thesis \cite[Theorem\,5.3]{Stein} 
generalizes all these investigations to the set
\[\Omega \,:=\, \big\{ \,\Phi_{2s_1},\Phi_{2s_2}^*,
\Psi_{2s_3},\Psi_{2s_4}^*\,|\, s_1,s_2,s_3,s_4 \in {\N}\,\big\} \] 
by specifying the subsets of $\Omega$ which are algebraically independent over ${\Q}$. 
In particular, all two-element subsets of $\Omega$ consist of algebraically
independent numbers over ${\Q}$. Also, 22 classes of three-element subsets 
are precisely described having the same property. Additionally, it turns out that any four
numbers in $\Omega$ are algebraically dependent over ${\Q}$, see \cite[Theorem\,5.4]{Stein}. \\
Although he has settled the problem of the algebraic independence 
or dependence for subsets of $\Omega$ completely, M.\,Stein posed 
the problem on the {\em linear\/} independence and dependence of $m$-subsets of 
$\Omega$ with $m>4$ (cf. \cite[Section\,5.3]{Stein}). 
He gives two partial answers. For his first result in this direction
he denotes for any positive integer $s$ by $W_{2s}\in \Omega$ 
one of the numbers $\Phi_{2s}$, $\Phi_{2s}^*$, $\Psi_{2s}$, or $\Psi_{2s}^*$. 
\begin{quote}
{\bf Theorem\,A} \cite[Theorem\,5.5]{Stein} \\ 
{\em  Let $1\leq s_1<s_2<\dots <s_m$ be $m$ positive integers for some $m\in {\N}$. 
Then the numbers $W_{2s_1},\dots,W_{2s_m}$ are linearly independent over 
${\Q}(E/\pi,k)$.\/}
\end{quote}
For the second result a positive integer $s$ is fixed, and the four numbers 
$\Phi_{2s}$, $\Phi_{2s}^*$, $\Psi_{2s}$, and $\Psi_{2s}^*$ are con- \\sidered. 
\begin{quote}
{\bf Theorem\,B} \cite[Theorem\,5.6]{Stein} \\
{\em For any $s\geq 2$ the four numbers $\Phi_{2s}$, $\Phi_{2s}^*$, $\Psi_{2s}$, 
and $\Psi_{2s}^*$ are linearly independent over ${\Q}$, i.e., the linear equation
\begin{equation}
t_s\Phi_{2s} + u_s\Phi_{2s}^* + v_s\Psi_{2s} + w_s\Psi_{2s}^* \,=\, 0
\label{EQ1.40}
\end{equation}
has no nontrivial solution $t_s,u_s,v_s,w_s\in {\Q}$. For $s=1$ 
the general solution of (\ref{EQ1.40}) is\/}
\begin{equation}
-2u_1\Phi_2 + u_1\Phi_2^* + u_1\Psi_2^* \,=\, 0 \qquad \big( u_1 \in {\Q}\big) \,.
\label{EQ1.50}
\end{equation}
\end{quote}
By setting $u_1=1$, M.\,Stein regains in (\ref{EQ1.50}) the result from (\ref{EQ1.30}). \\

In his endeavor to construct subsets of $\Omega$ with more than 
three linearly independent elements, M.\,Stein found two examples 
(c.f. \cite[p.\,83]{Stein}),
\[(-2u+v)\Phi_2 + u\Phi_2^* + (u-v)\Psi_2^* - 7v\Phi_4 
+ 8v\Phi_4^* + v\Psi_4 \,=\, 0 \qquad \big( u,v \in {\Q} \big) 
\]
and
\begin{eqnarray*}
&& -6(u+w)\Phi_2 + 6u\Phi_2^* + 6w\Psi_2^* + 6(u-v-w)\Phi_4 
+ 6v\Phi_4^* + 6(u-w)\Psi_4 + 32(8u-v-8w)\Phi_6 \\
&+& (-248u + 31v + 248w)\Phi_6^* + (-8u + v + 8w)\Psi_6^* \,=\, 0 \qquad 
\big( u,v,w \in {\Q} \big)\,. 
\end{eqnarray*}
Using a computer-algebra-system, M.\,Stein found more examples 
of such linear identities for larger subsets of $\Omega$. 
We define $V_{m}$ to be the set of all $\left(t_{1},\ldots,t_{4m}\right)\in\mathbb{Q}^{4m}$
satisfying
\begin{equation}
\sum_{s=1}^{m}\left(t_{4s-3}\Phi_{2s}+t_{4s-2}\Phi_{2s}^{*}
+t_{4s-1}\Psi_{2s}+t_{4s}\Psi_{2s}^{*}\right)=0.
\label{eq: defining equation of the vectorspace in question}
\end{equation}
\begin{conjecture}[M. Stein]
\label{conj: Stein}One has $\dim_{\mathbb{Q}}V_{m}=m$ for every
positive integer $m$. Moreover, we have for $\left(t_{1},\ldots,t_{4m}\right)\in V_{m}$
\begin{equation}
t_{4s}=0\qquad\text{for}\quad2\mid s,\qquad\text{and}\qquad t_{4s-1}
=0\qquad\text{for}\quad2\nmid s.
\label{eq: additional assertion in Stein Conjecture}
\end{equation}
\end{conjecture}
The goal of this paper is to prove this conjecture.
\begin{thm}\label{Thm1}
Conjecture \ref{conj: Stein} is true.
\end{thm}
For the proof of this theorem, we exploit many 
auxiliary results which are already provided 
by the theory sketched above for the transcendence and algebraic
independence results of the numbers from $\Omega$. 
All these tools are contained in \cite{Stein}, 
so that we can always refer the reader to this Ph.D. thesis. 
The following lines outline the key steps in our proof, 
and thus the structure of the present paper.
\paragraph*{Organization of the paper.}
The reasoning of our approach can be outlined as follows. 
Firstly, we make use of the known explicit formulae for a number 
$\xi \in \{ \Phi_{2s},\Phi_{2s}^*,
\Psi_{2s},\Psi_{2s}^* \}$. Indeed, $\xi$ can be written as a nonzero,
multivariate polynomial $P^{(\xi)}(E/\pi, K/\pi,k)$ 
in the algebraically independent numbers $E/\pi, K/\pi,k$ 
given by the complete elliptic integrals in (\ref{EQ1.20}). 
All of this is detailed in the Sections \ref{Sec2} and \ref{Sec3}.
Thereafter, we observe that $P^{(\xi)}(E/\pi, K/\pi,k)$ is essentially a linear form
in (precisely) one of four auxiliary polynomials 
$\Theta_j^{\pm},\Lambda_j^{\pm}$. Then, by degree considerations 
and by analysing the structure of $P^{(\xi)}(E/\pi, K/\pi,k)$ 
for all possible values of $\xi$, the problem of characterizing that
$\underline{t}\in V_m$ can be translated to investigating the vector space
of $(a,b,c,d)\in {\Q}^4$ satisfying
\[a\Theta_j^+ + b\Theta_j^- + c\Lambda_j^+ + d\Lambda_j^- \,=\, 0\,,\]
i.e., to determine linear dependencies between the auxiliary polynomials 
which is the content of Section \ref{Sec4}.
In Section \ref{Sec5}, we use the fact that $\Theta_j^{\pm},\Lambda_j^{\pm}$ 
originate from the Laurent series of the Jacobi elliptic functions, to show that 
determining the linear dependencies of $\Theta_j^{\pm},\Lambda_j^{\pm}$ 
can be reduced to proving an identity\footnote
{Said identity is seen, after some manipulations, to be equivalent 
to an analogue (cf. Lemma \ref{Lem: Elliptic function identity})
of the well-known doubling formula $\sin(2z) = 2\sin(z)\cos(z), z\in \mathbb{C},$
for the Jacobi elliptic function $\mathrm{sn}$.}
between the Jacobi elliptic functions.
After this intermediate steps are done, we can give the proof of Theorem \ref{Thm1} for 
all sufficiently large $m$, and deal with the finitely many remaining cases in the appendix.

\section{Notation and Preliminaries} \label{Sec2}
Let $s\geq 1$, and let $\xi$ denote one 
of the numbers from the set 
$\Omega_s :=\{ \Phi_{2s},\Phi_{2s}^*,\Psi_{2s},\Psi_{2s}^* \}$. 
We start by defining the notation needed to rewrite $\xi$ 
in a way suitable to our purposes, and we shall briefly describe 
the maxim of the rewriting. With the given algebraic number $\b$, 
which defines the sequences $U_n$ and $V_n$ for 
$\Phi_{2j},\Phi_{2j}^*,\Psi_{2j},\Psi_{2j}^*$,
we obtain an uniquely determined modulus $k$ from the identity
\[\beta^2 \,=\, \exp \Big( \,-\frac{\pi K(\sqrt{1-k^2})}{K(k)}\,\Big) \,,\]
cf. \cite[Sec.\,1.4]{Stein} and (\ref{EQ1.20}). Thus, $E,K,k$ are fixed. 
Next, we need the Jacobi elliptic functions (\cite{Byrd}). 
Let $sn(z,k)=sn(z)$ be the sine amplitude function obtained 
by inverting the meromorphic map
\[z \,\to \,\int_0^z \frac{dt}{\sqrt{(1-t^2)(1-k^2t^2)}} \,.\]
Then, using Glaisher's notation, let
\[\begin{array}{lcllcl}
\ns &:=& \displaystyle \frac{1}{\sn}\,,\quad & \nc &:=& \displaystyle \frac{1}{1-\sn} \,,\\ \\
\nd &:=& 1-k^2\sn \,,\quad & \dn &:=& \displaystyle \frac{1}{1-k^2\sn} \,.
\end{array} \]    
Moreover, we need the following power series expansions. 
\begin{eqnarray*}
\ns - \frac{1}{z^2} - \frac{1}{3}(1+k^2) &=& \sum_{j\geq 1} c_j(k)z^{2j} \,,\\
(1-k^2)\big( \nd -1 \big) &=& \sum_{j\geq 1} d_j(k)z^{2j} \,,\\
(1-k^2)\big( \nc -1 \big) &=& \sum_{j\geq 1} e_j(k)z^{2j} \,,\\
\dn -1 &=& \sum_{j\geq 1} f_j(k)z^{2j} \,.
\end{eqnarray*}
The coefficients $c_j,d_j,e_j,f_j$ are polynomials in the elliptic modulus $k^2$. 
Each such polynomial can be computed recursively; the recurrence formulae are
consequences of nonlinear differential equations satisfied by $\ns, \nd, \nc$, 
and $\dn$, cf. the Lemmata\,3.1 to 3.4 in \cite[p.\,17-20]{Stein}, or \cite{Elsner2}.    
Moreover, the following auxiliary polynomials
\begin{align}
\Theta_{j-1}^{-}\left(k\right) & := c_{j-1}\left(k\right)-d_{j-1}\left(k\right)
=:\sum_{i=0}^{j}\alpha_{j,i}k^{2i},\label{eq: definition of ThetaMinus}\\
\Theta_{j-1}^{+}\left(k\right) & := c_{j-1}\left(k\right)+d_{j-1}\left(k\right)
=:\sum_{i=0}^{j}\beta_{j,i}k^{2i},\label{eq: definition of ThetaPlus}\\
\Lambda_{j-1}^{-}\left(k\right) & := e_{j-1}\left(k\right)-f_{j-1}\left(k\right)
=:\sum_{i=0}^{j}\gamma_{j,i}k^{2i},\label{eq: definition of LambdaMinus}\\
\Lambda_{j-1}^{+}\left(k\right) & := e_{j-1}\left(k\right)+f_{j-1}\left(k\right)
=:\sum_{i=0}^{j}\delta_{j,i}k^{2i}\label{eq: definition of LambdaPlus}
\end{align}
for $j\geq2$ will play a fundamental role in stating the explicit formulae 
for $\xi \in \Omega_j$. We note that for $s\geq 1$ any 
$\xi\in\left\{ \Phi_{2s},\Phi_{2s}^{*},\Psi_{2s},\Psi_{2s}^{*}\right\} $ 
can be written as 
\begin{equation}
\xi=P^{\left(\xi\right)}\left(k,\frac{K}{\pi},\frac{E}{\pi}\right)
=\mathbf{P}_{I}^{\left(\xi\right)}
\left( k,\frac{K}{\pi}\right)
+\mathbf{P}_{II}^{\left(\xi\right)}\left(k,\frac{K}{\pi},\frac{E}{\pi}\right)
\label{eq: spliting of xi in in first polynomial and second}
\end{equation}
where the multivariate polynomials 
$\mathbf{P}_{I}^{\left(\xi\right)},\mathbf{P}_{II}^{\left(\xi\right)}$
are defined precisely in the subsequent section. 
The idea behind this decomposition of $\xi$ is that the second summand 
$P_{II}^{(\xi)}$ is a 
multivariate polynomial in $k,K/\pi,E/\pi$, 
which gathers four kinds of terms, namely, 
all terms that are either rational or are rational multiples of
\[{\Big( \,\frac{2K}{\pi}\,\Big)}^2\,,\quad {\Big( \,\frac{2K}{\pi}\,\Big)}^2(2k^2-1)\,,
\quad {\Big( \,\frac{2K}{\pi}\,\Big)}^2\Big( \,\frac{6E}{K}-5+4k^2\,\Big)\,,
\quad {\Big( \,\frac{2K}{\pi}\,\Big)}^2\Big( \,\frac{2E}{K} -1\,\Big)\,;\]
note that 
\[{\Big( \,\frac{2K}{\pi}\,\Big)}^2 \frac{E}{K} \,
=\, \Big( \,\frac{2K}{\pi}\,\Big) \Big( \,\frac{2E}{\pi}\,\Big) \,.\]
The first summand $P_{I}^{(\xi)}$ is also a multivariate polynomial, 
however in $k,K/\pi$ only, and gathers all rational multiples of the higher powers
\[{\Big( \,\frac{2K}{\pi}\,\Big)}^{2j}k^{2i} \,,\]
where $j\geq 2$ and $i\in \{ 0,1,\dots,j \}$. \\

Finally, we need the rational numbers $a_j$ and $b_j$ defined 
by the series expansions of the circular functions $\mbox{cosec}^2 z$ 
and $\mbox{sec}^2 z$, 
respectively, given by
\[\mbox{cosec}^2 z \,=\, \frac{1}{z^2} + \sum_{j\geq 0} a_jz^{2j}\,,\qquad 
a_j \,=\, \frac{{(-1)}^j (2j+1)2^{2j+2}B_{2j+2}}{(2j+2)!} \quad (j\geq 0) \,,\]
and
\[\mbox{sec}^2 z \,=\, \sum_{j\geq 0} b_jz^{2j}\,,\qquad 
b_j \,=\, \frac{{(-1)}^j (2j+1)2^{2j+2}(2^{2j+2}-1)B_{2j+2}}{(2j+2)!} \quad (j\geq 0) \,.\]
Here, $B_2=1/6$, $B_4=-1/30$, $B_6=1/42$, \dots denote the Bernoulli numbers. 
The explicit formulae for $\xi$ involve coefficients $\sigma_i(s)$, which are the 
elementary symmetric functions of the $(s-1)$ numbers $-1^2,-2^2,\dots,-{(s-1)}^2$ 
for $s\geq 2$. They are defined by $\sigma_0(s):=1$, 
and for $i\geq 1$ and $j=1,\dots,s-1$ by
\[\sigma_{i}\left(s\right) \,:=\, {(-1)}^i 
\sum_{1\leq r_{1}<\ldots<r_{i}\leq s-1}r_{1}^{2}\ldots r_{i}^{2} \,.\]
Now, we state the explicit formulae for $\xi$ in terms of $k,K/\pi$, 
and $E/\pi$, cf. \cite[Sec.\,3.2]{Stein}. 
\begin{eqnarray}
\Phi_{2s} &=& \frac{1}{(2s-1)!}\left[ \,-\frac{{(s-1)!}^2}{24}
\left( \,1 - {\Big( \,\frac{2K}{\pi}\,\Big)}^2\Big( \,\frac{6E}{K}-5+4k^2\,\Big)\,
\right) \right. \nonumber \\
&&\left. +\,\sum_{j=1}^{s-1} \sigma_{s-j-1}(s) 
\frac{{(-1)}^j(2j)!}{2^{2j+3}}\Big( \,a_j - {\Big( \,
\frac{2K}{\pi}\,\Big)}^{2j+2}\Theta_j^-(k)\,\Big) \,\right]
\qquad \mbox{(for $s$ even)}\,,
\label{Phi.even} \\
\Phi_{2s} &=& \frac{1}{(2s-1)!}\left[ \,\frac{{(s-1)!}^2}{24}
\left( \,1 - {\Big( \,\frac{2K}{\pi}\,\Big)}^2\big( 1 - 2k^2 \big)\,\right) \right. 
\nonumber \\
&&\left. +\,\sum_{j=1}^{s-1} \sigma_{s-j-1}(s) 
\frac{{(-1)}^j(2j)!}{2^{2j+3}}\Big( \,a_j - {\Big( \,
\frac{2K}{\pi}\,\Big)}^{2j+2}\Theta_j^+(k)\,\Big) \,\right]
\qquad \mbox{(for $s$ odd)}\,,
\label{Phi.odd} 
\end{eqnarray}
\begin{eqnarray}
\Phi_{2s}^* &=& \frac{1}{(2s-1)!}\left[ \,
\frac{{(s-1)!}^2}{24}\left( \,1 - {\Big( \,
\frac{2K}{\pi}\,\Big)}^2\big( 1 - 2k^2 \big)\,\right) \right. 
\nonumber \\
&&\left. -\,\sum_{j=1}^{s-1} \sigma_{s-j-1}(s) 
\frac{{(-1)}^j(2j)!}{2^{2j+3}}\Big( \,a_j - {\Big( \,
\frac{2K}{\pi}\,\Big)}^{2j+2}\Theta_j^+(k)\,\Big) \,\right]
\qquad \mbox{(for $s$ even)}\,,
\label{Phi*.even} \\
\Phi_{2s}^* &=& \frac{1}{(2s-1)!}\left[ \,-\frac{{(s-1)!}^2}{24}
\left( \,1 - {\Big( \,\frac{2K}{\pi}\,\Big)}^2\Big( \,
\frac{6E}{K}-5+4k^2\,\Big)\,\right) \right. 
\nonumber \\
&&\left. -\,\sum_{j=1}^{s-1} \sigma_{s-j-1}(s) \frac{{(-1)}^j(2j)!}{2^{2j+3}}
\Big( \,a_j - {\Big( \,\frac{2K}{\pi}\,\Big)}^{2j+2}\Theta_j^-(k)\,
\Big) \,\right]
\qquad \mbox{(for $s$ odd)}\,,
\label{Phi*.odd} \\
\Psi_{2s} &=& \frac{1}{(2s-1)!}\left[ \,-\frac{{(s-1)!}^2}{8}\left( \,
1 + {\Big( \,\frac{2K}{\pi}\,\Big)}^2\Big( \,1 - \frac{2E}{K}\,\Big)\,\right) \right. 
\nonumber \\
&&\left. +\,\sum_{j=1}^{s-1} \sigma_{s-j-1}(s) 
\frac{{(-1)}^j(2j)!}{2^{2j+3}}\Big( \,b_j - {\Big( \,
\frac{2K}{\pi}\,\Big)}^{2j+2}\Lambda_j^-(k)\,\Big) \,\right]
\qquad \mbox{(for $s$ even)}\,,
\label{Psi.even} \\
\Psi_{2s} &=& \frac{1}{(2s-1)!}\left[ \,\frac{{(s-1)!}^2}{8} 
\left( \,{\Big( \,\frac{2K}{\pi}\,\Big)}^2 - 1\,\right) \right. 
\nonumber \\
&&\left. +\,\sum_{j=1}^{s-1} \sigma_{s-j-1}(s) 
\frac{{(-1)}^j(2j)!}{2^{2j+3}}\Big( \,{\Big( \,\frac{2K}{\pi}\,\Big)}^{2j+2}
\Lambda_j^+(k) - b_j\,\Big) \,\right]
\qquad \mbox{(for $s$ odd)}\,,
\label{Psi.odd} \\
\Psi_{2s}^* &=& \frac{1}{(2s-1)!}\left[ \,-\frac{{(s-1)!}^2}{8} 
\left( \,{\Big( \,\frac{2K}{\pi}\,\Big)}^2 - 1\,\right) \right. 
\nonumber \\
&&\left. +\,\sum_{j=1}^{s-1} \sigma_{s-j-1}(s) 
\frac{{(-1)}^j(2j)!}{2^{2j+3}}\Big( \,{\Big( \,\frac{2K}{\pi}\,\Big)}^{2j+2}
\Lambda_j^+(k) - b_j\,\Big) \,\right]
\qquad \mbox{(for $s$ even)}\,,
\label{Psi*.even} \\
\Psi_{2s}^* &=& \frac{1}{(2s-1)!}\left[ \,
\frac{{(s-1)!}^2}{8}\left( \,1 + {\Big( \,\frac{2K}{\pi}\,\Big)}^2
\Big( \,1 - \frac{2E}{K}\,\Big)\,\right) \right. 
\nonumber \\
&&\left. +\,\sum_{j=1}^{s-1} \sigma_{s-j-1}(s) 
\frac{{(-1)}^j(2j)!}{2^{2j+3}}\Big( \,b_j - {\Big( \,
\frac{2K}{\pi}\,\Big)}^{2j+2}\Lambda_j^-(k)\,\Big) \,\right]
\qquad \mbox{(for $s$ odd)}\,.
\label{Psi*.odd}
\end{eqnarray}

\newpage

\section{Definition of $\mathbf{P}_{I}^{\left(\xi\right)}$ and $\mathbf{P}_{II}^{\left(\xi\right)}$} \label{Sec3}
Let us define $\mathbf{P}_{I}^{\left(\xi\right)}$ first. Fix $\xi\in\left\{ \Phi_{2s},\Phi_{2s}^{*},\Psi_{2s},\Psi_{2s}^{*}\right\} $,
for $s\geq1$. Then,
\begin{equation}
\mathbf{P}_{I}^{\left(\xi\right)}\left(k,\frac{K}{\pi}\right):=\sum_{j=1}^{s-1}\left(
\frac{2K}{\pi}\right)^{2j+2}w_{j}^{\left(s\right)}P_{j}^{\left(\xi\right)}\left(k\right)\,,\label{eq: shape of the first auxiliary polynomial}
\end{equation}
where $w_{j}^{\left(s\right)}$ is the rational ``weight factor''
\[w_{j}^{\left(s\right)}:= \left(-1\right)^{j}\frac{\sigma_{s-j-1}\left(s\right)}{2^{2j+3}\left(2s-1\right)!}\left(2j\right)! \qquad (j=1,\dots,s-1)\,.\]
$P_{j}^{\left(\xi\right)}\in\mathbb{Q}\left[x\right]$ is given by 
\begin{align}
{\bf P}_j^{(s)} \,:=\,
\left(P_{j}^{\left(\Phi_{2s}\right)},P_{j}^{\left(\Phi_{2s}^{*}\right)},P_{j}^{\left(\Psi_{2s}\right)},P_{j}^{\left(\Psi_{2s}^{*}\right)}\right):=\begin{cases}
\big( -\Theta_{j}^{-},\Theta_{j}^{+},-\varLambda_{j}^{-},\varLambda_{j}^{+}\,\big) & \mbox{for}\,s\,\mbox{even},\\
\big( -\Theta_{j}^{+},\Theta_{j}^{-},\varLambda_{j}^{+},-\varLambda_{j}^{-}\,\big) & \mbox{for}\,s\,\mbox{odd}
\end{cases}\label{eq: definition of the auxiliary polynomials}
\end{align}
for $j\in\left\{ 1,\ldots,s-1\right\} $. Let us now define $\mathbf{P}_{II}^{\left(\xi\right)}$.
Letting 
\[
\hat{w}_{s}:=\frac{\left(s-1\right)!^{2}}{24\left(2s-1\right)!}
\]
denote another rational weight factor, we have
\[
\mathbf{P}_{II}^{\left(\xi\right)}\left(k,\frac{K}{\pi},\frac{E}{\pi}\right):=\hat{w}_{s}\left(\frac{2K}{\pi}\right)^{2}P_{0}^{\left(\xi\right)}\left(k,
\frac{K}{\pi},\frac{E}{\pi}\right)+R^{\left(\xi\right)}
\]
where $P_{0}^{\left(\xi\right)}$ is the ``initial'' polynomial
given by
\begin{align}
\left(P_{0}^{\left(\Phi_{2s}\right)},P_{0}^{\left(\Phi_{2s}^{*}\right)},
P_{0}^{\left(\Psi_{2s}\right)},P_{0}^{\left(\Psi_{2s}^{*}\right)}\right) & :=\begin{cases}
\left(\frac{6E}{K}-5+4k^{2},\,2k^{2}-1,\,\frac{6E}{K}-3,\,-3\right) & 
\mathrm{for\,}s\,\mathrm{even},\\
\left(2k^{2}-1,\,\frac{6E}{K}-5+4k^{2},\,3,\,3-\frac{6E}{K}\right) & 
\mathrm{for\,}s\,\mathrm{odd}.
\end{cases}\label{eq: initial polynomial of Phi of s}
\end{align}
Moreover, $R^{\left(\xi\right)}$ is the rational number 
\begin{align}
R^{\left(\Phi_{2s}\right)} & :=\left(-1\right)^{s+1}\hat{w}_{s}
+\sum_{j=1}^{s-1}a_{j}w_{j}^{\left(s\right)}=:-R^{\left(\Phi_{2s}^{*}\right)},
\label{eq: rational remainder of Phi and Phi star of s}\\
R^{\left(\Psi_{2s}\right)} & :=-3\hat{w}_{s}+\left(-1\right)^{s}
\sum_{j=1}^{s-1}b_{j}w_{j}^{\left(s\right)}=:-R^{\left(\Psi_{2s}^{*}\right)}
\label{eq: rational remainder of Psi and Psi star of s} \,.
\end{align}

\section{Connection to linear dependencies of $\Theta_{j}^{\pm},\Lambda_{j}^{\pm}$} \label{Sec4}
In this section, we demonstrate how the 
linear dependencies in (\ref{eq: defining equation of the vectorspace in question})
translate to linear dependencies of the auxiliary polynomials 
$\Theta_{j}^{\pm},\Lambda_{j}^{\pm}$.
We define 
\begin{align*}
x_{1} & :=1, & x_{2} & :=\left(\frac{2K}{\pi}\right)^{2},\\
x_{3} & :=\left(\frac{2K}{\pi}\right)^{2}\left(2k^{2}-1\right), 
& x_{4} & :=\left(\frac{2K}{\pi}\right)^{2}\left(\frac{6E}{K}-5+4k^{2}\right),
\end{align*}
and, for $2\leq j\leq m$, we let 
\begin{align*}
x_{4j-3} & :=\left(\frac{2K}{\pi}\right)^{2j} \cdot 
\begin{cases}
	\Theta_{j-1}^{-}\left(k\right) & \mathrm{for\,}j\,\mathrm{even},\\
	\Theta_{j-1}^{+}\left(k\right) & \mathrm{for\,}j\,\mathrm{odd},
\end{cases} 
& x_{4j-2} & :=\left(\frac{2K}{\pi}\right)^{2j} \cdot 
\begin{cases}
	\Theta_{j-1}^{+}\left(k\right) & \mathrm{for\,}j\,\mathrm{even},\\
	\Theta_{j-1}^{-}\left(k\right) & \mathrm{for\,}j\,\mathrm{odd},
\end{cases}\\
x_{4j-1} & :=\left(\frac{2K}{\pi}\right)^{2j} \cdot 
\begin{cases}
	\Lambda_{j-1}^{-}\left(k\right) & \mathrm{for\,}j\,\mathrm{even},\\
	\Lambda_{j-1}^{+}\left(k\right) & \mathrm{for\,}j\,\mathrm{odd},
\end{cases} & x_{4j} & :=\left(\frac{2K}{\pi}\right)^{2j} \cdot 
\begin{cases}
	\Lambda_{j-1}^{+}\left(k\right) & \mathrm{for\,}j\,\mathrm{even},\\
	\Lambda_{j-1}^{-}\left(k\right) & \mathrm{for\,}j\,\mathrm{odd}.
\end{cases}
\end{align*}
For the ease of exposition, we put 
$\mathbf{P}_{I}^{\left(\xi\right)}:=\mathbf{P}_{I}^{\left(\xi\right)}
\left(k,\frac{K}{\pi}\right)$
and $\mathbf{P}_{II}^{\left(\xi\right)}:=\mathbf{P}_{II}^{\left(\xi\right)}
\left(k,\frac{K}{\pi},\frac{E}{\pi}\right)$
for any $\xi\in\left\{ \Phi_{2s},\Phi_{2s}^{*},\right.$ $\left.\Psi_{2s},\Psi_{2s}^{*}\right\} $.
We restate (\ref{eq: defining equation of the vectorspace in question})
by collecting all terms involving $x_{1},\ldots,x_{4}$, and thus
conclude via the algebraic independence of $k,\frac{K}{\pi},\frac{E}{\pi}$
over $\mathbb{Q}$ that
\begin{equation}
\sum_{i=1}^{m}\left(\mathbf{P}_{II}^{\left(\Phi_{2i}\right)}t_{4i-3}
+\mathbf{P}_{II}^{\left(\Phi_{2i}^{*}\right)}t_{4i-2}
+\mathbf{P}_{II}^{\left(\Psi_{2i}\right)}t_{4i-1}
+\mathbf{P}_{II}^{\left(\Psi_{2i}^{*}\right)}t_{4i}\right)=0.
\label{eq: vorspann}
\end{equation}
Moreover, we rearrange the remaining part of 
(\ref{eq: defining equation of the vectorspace in question}),
namely
\[
\sum_{i=2}^{m}\left(\mathbf{P}_{I}^{\left(\Phi_{2i}\right)}t_{4i-3}
+\mathbf{P}_{I}^{\left(\Phi_{2i}^{*}\right)}t_{4i-2}
+\mathbf{P}_{I}^{\left(\Psi_{2i}\right)}
t_{4i-1}+\mathbf{P}_{I}^{\left(\Psi_{2i}^{*}\right)}t_{4i}
\right),
\]
first to
\[
\sum_{i=2}^{m}\sum_{j=2}^{i}
\left(
\frac{2K}{\pi}\right)^{2j}w_{j-1}^{\left(i\right)}
\left(P_{j-1}^{\left(\Phi_{2i}\right)}\left(k\right)t_{4i-3}+P_{j-1}^{\left(\Phi_{2i}^{*}\right)}
\left(k\right)t_{4i-2}+P_{j-1}^{\left(\Psi_{2i}\right)}
\left(k\right)t_{4i-1}+P_{j-1}^{\left(\Psi_{2i}^{*}\right)}\left(k\right)t_{4i}
\right).
\]
Then, by interchanging 
the order of summation, we deduce that
\[
\sum_{j=2}^{m}\sum_{i=j}^{m}\left(\frac{2K}{\pi}\right)^{2j}
w_{j-1}^{\left(i\right)}\left(P_{j-1}^{\left(\Phi_{2i}\right)}\left(k\right)t_{4i-3}
+P_{j-1}^{\left(\Phi_{2i}^{*}\right)}
\left(k\right)t_{4i-2}+P_{j-1}^{\left(\Psi_{2i}\right)}\left
(k\right)t_{4i-1}+P_{j-1}^{\left(\Psi_{2i}^{*}\right)}
\left(k\right)t_{4i}\right)=0.
\]
Again, by applying the algebraic independence of $k,\frac{K}{\pi},\frac{E}{\pi}$
over $\mathbb{Q}$, we infer, for $2\leq j\leq m$, that 
\begin{equation}
\sum_{i=j}^{m}\left(\frac{2K}{\pi}\right)^{2j}
w_{j-1}^{\left(i\right)}\left(P_{j-1}^{\left(\Phi_{2i}\right)}\left(k\right)t_{4i-3}
+P_{j-1}^{\left(\Phi_{2i}^{*}\right)}\left(k\right)t_{4i-2}
+P_{j-1}^{\left(\Psi_{2i}\right)}\left(k\right)t_{4i-1}
+P_{j-1}^{\left(\Psi_{2i}^{*}\right)}\left(k\right)t_{4i}\right)
=0.
\label{eq: hauptteil}
\end{equation}
Note that the Equations (\ref{eq: vorspann}) and (\ref{eq: hauptteil})
can be regarded as linear forms in $x_{1},\ldots,x_{4m}$.
Now, we define the matrix 
$A:=\left(a_{i,j}\right)\in\mathbb{C}^{\left(m+3\right)\times\left(4m\right)}$
in which each entry $a_{i,j}$ represents the factor in front of $x_{i}t_{j}$
from the Equations (\ref{eq: vorspann}) and (\ref{eq: hauptteil}).
\begin{rem}
\label{rem: number of zeros in the beginning of a line}For $i\geq5$,
the $i$-th row of $A$ starts with $4\left(i-4\right)$ many zeros. 
For future reference, we record that the $4\times 4$ submatrix 
$(a_{i,j})_{1\leq i,j\leq 4}$ of rank three is given by 
\begin{equation} \frac{1}{24}
		 \left( 
			\begin{array}{rrrr}
				1 & -1 & -3 &3 \\ 
				0 & 0 & 3 & 0\\ 
				1& 0 &0& 2\\ 
				0& 1 & 0& -1\\ 
			\end{array}
		\right) \label{matrix vorspann}.
\end{equation}
\end{rem}
The key issue for our proof of Conjecture \ref{conj: Stein} is to
analyse the kernel of $A$. For this purpose, we exhibit a ``quasi-periodicity''
property of the rows of $A$ in the subsequent lemma. For illustrating
the underlying structures, the reader can find an example, where $m=3$,
in the final section.
\begin{lem}
Let $m\geq3$, consider in the matrix $A$ four consecutive elements in the $\nu$-th
row, $v=1,\ldots,m-2$, as one vector, and let $l=1,\ldots,m-\nu-1$.
Then, there is a permutation $\pi=\pi_{\nu,l}$ of the numbers 
$\left\{ 4\left(\nu+l\right)+1,\ldots,4\left(\nu+l\right)+4\right\} $
such that the vectors $\left(a_{4+\nu,4\nu+1},a_{4+\nu,4\nu+2},
a_{4+\nu,4\nu+3},a_{4+\nu,4\nu+4}\right),$
and
\[
\left(a_{4+\nu,\pi\left(4\left(\nu+l\right)+1\right)},
a_{4+\nu,\pi\left(4\left(\nu+l\right)+2\right)},
a_{4+\nu,\pi\left(4\left(\nu+l\right)+3\right)},
a_{4+\nu,\pi\left(4\left(\nu+l\right)+4\right)}\right)
\]
are linearly dependent over $\mathbb{Q}$.
\label{Lem1}
\end{lem}
\begin{proof}
Because $4+\nu\geq5$, it suffices to consider only the $a_{i,j}$
arising from (\ref{eq: hauptteil}). We observe that each 
$\Theta_{3+\nu}^{+},\Theta_{3+\nu}^{-},\Lambda_{3+\nu}^{+},\Lambda_{3+\nu}^{-}$
occurs exactly once in the term 
\[
P_{j-1}^{\left(\Phi_{2i}\right)}\left(k\right)t_{4i-3}
+P_{j-1}^{\left(\Phi_{2i}^{*}\right)}\left(k\right)t_{4i-2}
+P_{j-1}^{\left(\Psi_{2i}\right)}\left(k\right)
t_{4i-1}+P_{j-1}^{\left(\Psi_{2i}^{*}\right)}\left(k\right)t_{4i}
\]
for any $j=4+\nu$ and any $i=4+\nu,\ldots,m$. This defines the permutation
$\pi_{\nu,l}$, and since the weight factor $w_{j-1}^{\left(i\right)}\neq0$
from (\ref{eq: hauptteil}) does not depend on 
$\xi\in\left\{ \Phi_{2i},\Phi_{2i}^{*},\Psi_{2i},\Psi_{2i}^{*}\right\} $,
the lemma follows.
\end{proof}

\section{Linear dependencies of $\Theta_{j}^{\pm},\Lambda_{j}^{\pm}$} \label{Sec5}
Due to Lemma 3.6 in the Ph.D. thesis \cite{Stein} of M. Stein, there are explicit
formulae for some coefficients of the auxiliary polynomials. We shall
use these formulae. 
\begin{lem}
\label{lem: explicit expressions for first coefficients of auxilliary poly}Let
$\Theta_{j}^{\pm},\,\Lambda_{j}^{\pm}$ be as in (\ref{eq: definition of ThetaMinus})
to (\ref{eq: definition of LambdaPlus}). Then,
\begin{align*}
\alpha_{j,0} & =a_{j-1}=\beta_{j,0},\qquad\gamma_{j,0}=b_{j-1}=\delta_{j,0}.
\end{align*}
By putting $\kappa_{j-1}:=\frac{\left(-1\right)^{j-1}2^{2j-3}}{\left(2j-2\right)!}$,
the coefficients of the quadratic terms are given by 
\begin{align*}
\alpha_{j,1} & =\kappa_{j-1}-\frac{j}{2}a_{j-1}=\beta_{j,1}+2\kappa_{j-1},\\
\gamma_{j,1} +2\kappa_{j-1} & =\kappa_{j-1}-\frac{j}{2}b_{j-1}=\delta_{j,1},
\end{align*}
the coefficients of the quartic terms, letting 
$\hat{\kappa}_{j-1}:=\frac{j\left(4j-7\right)}{32}$,
can be written as 
\begin{align*}
\alpha_{j,2} & =\frac{\kappa_{j-1}}{16}(7-8j-2^{2j-1})+\hat{\kappa}_{j-1}a_{j-1},\\
\beta_{j,2} & =\frac{\kappa_{j-1}}{16}(-9+8j+2^{2j-1})+\hat{\kappa}_{j-1}a_{j-1},\\
\gamma_{j,2} & =\frac{\kappa_{j-1}}{16}(-7+8j-2^{2j-1})+\hat{\kappa}_{j-1}b_{j-1},\\
\delta_{j,2} & =\frac{\kappa_{j-1}}{16}(9-8j+2^{2j-1})+\hat{\kappa}_{j-1}b_{j-1},
\end{align*}
and the coefficients at $z^{2j}$ are given by 
\[
\alpha_{j,j}=2^{2j}a_{j-1},\qquad\beta_{j,j}
=\left(2-2^{2j}\right)a_{j-1},\qquad\gamma_{j,j}=\delta_{j,j}
=0.
\]
\end{lem}
\begin{lem}
\label{lem: subspace of linear dependencies of auxiliary polynomials} For
$j\geq1$, the subspace
\[
S_{j}:= 
\{ 
	\underline{t}:=\left(t_{1},\ldots,t_{4}\right)\in\mathbb{Q}^{4}:\,
	\langle 
		\underline{t},
		(-\Theta_{j}^{-},\Theta_{j}^{+},-\Lambda_{j}^{-}\Lambda_{j}^{+})
	\rangle =0
\} 
\]
is spanned by $\underline{v}_{j}:=\left(1-2^{2j+1},2^{2j+1},1,0\right)^{T}$.
In particular,\textcolor{black}{{} }for each even $j\geq2$\textcolor{black}{{}
the three quantities }$x_{4j-3}$, $x_{4j-2},x_{4j-1}$, and $x_{4j-3},\,x_{4j-2},\,x_{4j}$
for odd $j\geq 3$, respectively, are linearly dependent over $\mathbb{Q}$.\textcolor{red}{}
\end{lem}
\begin{proof}
The argument splits into two parts. Firstly, we show that any element
of $S_{j}$ is necessarily a rational multiple of $\underline{v}_{j}$,
and then that, indeed, any such multiple is an element of $S_{j}$.

(i) If $\underline{t}\in S_{j}$, then $\underline{t}$ is in the
kernel of the matrix
\[
\varXi:=\begin{pmatrix}-\alpha_{j+1,0} & \beta_{j+1,0} & -\gamma_{j+1,0} & \delta_{j+1,0}\\
-\alpha_{j+1,1} & \beta_{j+1,1} & -\gamma_{j+1,1} & \delta_{j+1,1}\\
-\alpha_{j+1,2} & \beta_{j+1,2} & -\gamma_{j+1,2} & \delta_{j+1,2}\\
-\alpha_{j+1,j+1} & \beta_{j+1,j+1} & -\gamma_{j+1,j+1} & \delta_{j+1,j+1}
\end{pmatrix}.
\]
Using the explicit expressions provided 
by Lemma \ref{lem: explicit expressions for first coefficients of auxilliary poly}
for the matrix entries, a calculation (performed by Mathematica) yields
that $\Xi$ can be transformed to the following row echelon form 
\[
\begin{pmatrix}1 & 0 & 2^{2j+1}-1 & 0\\
0 & 1 & -2^{2j+1} & 0\\
0 & 0 & 0 & 1\\
0 & 0 & 0 & 0
\end{pmatrix}.
\]
Therefore, the kernel of $\Xi$ is given by 
$\mathbb{Q}(1-2^{2j+1},\,2^{2j+1},\,1,\,0)$.
\\
\\
(ii) Let $\underline{t}:= t(1-2^{2j+1},\,2^{2j+1},\,1,\,0)$
for $t\in\mathbb{Q}$. Now we proceed to show that the Euclidean scalar product 
$\langle \underline{t},(\Theta_{j}^{-},\Theta_{j}^{+},
\varLambda_{j}^{-},\varLambda_{j}^{+})\rangle $
is indeed the zero polynomial in $k$. This is equivalent to demonstrating
that 
\begin{align}
0 & =(2^{2j+1}-1)\Theta_{j}^{-}+2^{2j+1}\Theta_{j}^{+}-\varLambda_{j}^{-}=(2^{2j+2}-1)c_{j}+d_{j}-e_{j}+f_{j} \label{eq: relation between c_j, d_j, e_j, f_j}
\end{align}
holds for any $j\geq1$. Multiplying the last equation by $z^{2j}$
and summing over $j\geq1$, this assertion is seen to be equivalent
to the identity
\begin{align*}
\left(1-k^{2}\right)\left(nc^{2}\left(z\right)-1\right) & 
=4\left[\mathrm{ns}^{2}\left(2z\right)-\frac{1}{(2z)^{2}}-\frac{1}{3}(1+k^{2})\right]\\
 & -\left[\mathrm{ns}^{2}\left(z\right)-\frac{1}{z^{2}}-\frac{1}{3}(1+k^{2})\right]\\
 & +\left(1-k^{2}\right)\left(\mathrm{nd}^{2}\left(z\right)-1\right)
 +\mathrm{dn}^{2}\left(z\right)-1.
\end{align*}
Simplifying the equation above, we conclude that it suffices to prove
\begin{equation}
\left(1-k^{2}\right)\mathrm{nc}^{2}\left(z\right)=4\mathrm{ns}^{2}\left(2z\right)-\mathrm{ns}^{2}\left(z\right)-\left(1+k^{2}\right)
+\left(1-k^{2}\right)\mathrm{nd}^{2}\left(z\right)-\mathrm{dn}^{2}\left(z\right)-1\label{eq: relation of nc^2, ns^2, sn^2, nd^2, 1, k^2}
\end{equation}
for demonstrating (\ref{eq: relation between c_j, d_j, e_j, f_j}). However,
this identity holds by the subsequent Lemma\,\ref{Lem: Elliptic function identity}.
\end{proof}

\begin{lem} \label{Lem: Elliptic function identity}
Let $k^2\in \C\setminus [1,\infty)$ and $z\in {\C}\setminus \{ mK\,:\, m\in {\Z} \}$. Then we have the identity
\[4\,\mathrm{ns}^2(2z) \,=\, (1-k^2)\big( \mathrm{nc}^2(z) - \mathrm{nd}^2(z) \big) 
+ \big( \mathrm{ns}^2(z) - \mathrm{dn}^2(z) \big) + (2+k^2) \,.\]
\end{lem}
\begin{proof} 
We distinguish two cases according whether $k$ vanishes or not. \\
{\em Case\,1.\/} $k\not= 0$. \\
We start with an algebraic identity which holds for all complex numbers $\alpha$ and $\b$,
\[{\big( \alpha - {(1-\beta)}^2 \big)}^2 \,=\, \alpha^2\beta^2 + (1-\beta)(\alpha -1 +\beta)\big( \alpha - 1 - \beta^2 +(2+\alpha)\beta \big) \,.\]
For a simpler notation we omit throughout the proof of Lemma\,\ref{Lem: Elliptic function identity} the parentheses around the argument $z$ from each 
elliptic function. 
Replacing $\alpha$ by $k^2$ and $\beta$ by $\dno$, we obtain
\begin{eqnarray}
&& {\big( k^2 - {(1-\dno)}^2 \big)}^2 \nonumber \\
&=& k^4\ddno + (1-\dno)(k^2-1+\dno)\big( (k^2-1)- \ddno + (2+k^2)\dno \big) \,.
\label{20}
\end{eqnarray}
From formula (121.00) in \cite[p.\,20]{Byrd} we recall the basic identities 
\begin{eqnarray}
\sno + \cno &=& 1 \,,
\label{H1} \\
k^2\sno + \dno &=& 1 \,.
\label{H2}
\end{eqnarray}
They imply the formulae
\begin{equation}
{(1-\dno)}^2 \,=\, k^4\ssno 
\label{30}
\end{equation}
and
\begin{equation}
(1-\dno)(k^2-1+\dno) \,=\, k^4\sno\, (1-\sno) \,=\, k^4\sno\, \cno \,.
\label{40}
\end{equation}
Substituting (\ref{30}) and (\ref{40}) into (\ref{20}) and dividing the resulting identity by $k^4$, it turns out that
\begin{eqnarray*}
&& {\big( 1 - k^2\ssno \big)}^2 \\
&=& \ddno + (k^2-1)\sno\, \cno - \sno\, \cno\, \ddno + (2+k^2)\sno\, \cno\, \dno \,.
\end{eqnarray*}
Using again (\ref{H1}) and (\ref{H2}), we continue our calculations by
\begin{eqnarray*}
{\big( 1 - k^2\ssno \big)}^2 &=& \big( \sno + \cno \big) \dno + (k^2-1)\sno\, \cno - \sno\, \cno\, \ddno \\
&& -\, k^2\sno\, \dno + (2+k^2)\sno\, \cno\, \dno \\
&=& (1-k^2)\sno \big( \dno - \cno \big) + \cno\, \dno\, \big( 1 - \sno\, \dno\, \big) \\
&& +\,(2+k^2)\sno\, \cno\, \dno \,.
\end{eqnarray*}
Since $z$ is no integer multiple of $K$, 
the number $\sno\, \cno\, \dno$ does not vanish. Thus, we obtain the identity
\[\frac{{\big( 1 - k^2\ssno \big)}^2}{\sno\, \cno\, \dno} \,
=\, (1-k^2)\Big( \,\frac{1}{\cno} - \frac{1}{\dno}\,\Big) 
+ \Big( \,\frac{1}{\sno} - \dno\,\Big) + (2+k^2) \,.\]
The left-hand side equals to $4\nss$, which follows from formula 
(124.01) in \cite[p.\,24]{Byrd}. Hence, the identity of the lemma is proven for $k\not= 0$. \\
{\em Case\,2.\/} $k=0$. \\
We have $\mbox{sn\,}z=\sin z$, $\mbox{cn\,}z=\cos z$ 
and $\mbox{dn\,}z \equiv 1$ (cf. formula (122.08) in \cite[p.\,21]{Byrd}). 
Therefore, in the case $k=0$, 
the identity in question reduces to the trigonometric formula
\[\frac{4}{\sin^2 (2z)} \,=\, \frac{1}{\cos^2 z} +  \frac{1}{\sin^2 z} \,.\]
Obviously this is a consequence of $\sin (2z)=2\sin z\, \cos z$, 
provided that $z$ is no integer multiple of $K=K(0)=\pi/2$.
\end{proof}

\section{Proof of Theorem \ref{Thm1}} \label{Sec6}
\begin{proof}
We introduce the $4\times 4$ matrices
\[{\bf R}^{(2s)} \,:=\, \left( \begin{array}{cccc} 
R^{(\Phi_{2s})} & R^{(\Phi_{2s}^*)} & R^{(\Psi_{2s})} & R^{(\Psi_{2s}^*)} \\
0 & 0 & 0 & -3{\hat w}_s \\
0 & {\hat w}_s & -2{\hat w}_s & 0 \\
{\hat w}_s & 0 & {\hat w}_s & 0
\end{array} \right) \,,\]
if $s$ is even, and
\[{\bf R}^{(2s)} \,:=\, \left( \begin{array}{cccc} 
R^{(\Phi_{2s})} & R^{(\Phi_{2s}^*)} & R^{(\Psi_{2s})} & R^{(\Psi_{2s}^*)} \\
0 & 0 & 3{\hat w}_s & 0 \\
{\hat w}_s & 0 & 0 & 2{\hat w}_s \\
0 & {\hat w}_s & 0 & -{\hat w}_s
\end{array} \right) \]
for odd $s$. Moreover, for $j=1,\dots,m-1$, we set
\[{\bf P}_j^{(s)} \,:=\quad \left\{ \begin{array}{ll}
\big( -\Theta_j^-,\,\Theta_j^+,\,-\Lambda_j^-,\,\Lambda_j^+ \big) 
& \quad \mbox{for $s$ even}\,, \\ \\
\big( -\Theta_j^+,\,\Theta_j^-,\,\Lambda_j^+,\,-\Lambda_j^- \big) 
& \quad \mbox{for $s$ odd}\,.
\end{array} \right. \]
The subsequent $(m+3)\times (4m)$ - matrix of central importance in our proof is
\[A_0^{(m)} \,:=\, \left( \begin{array}{cccccc}
{\bf R}^{(2)} & {\bf R}^{(4)} & {\bf R}^{(6)} & {\bf R}^{(8)} & \dots & {\bf R}^{(2m)} \\ \\
& w_1^{(2)}{\bf P}_1^{(2)} & w_1^{(3)}{\bf P}_1^{(3)} & w_1^{(4)}{\bf P}_1^{(4)} 
& \dots & w_1^{(m)}{\bf P}_1^{(m)} \\ \\
& & w_2^{(3)}{\bf P}_2^{(3)} & w_2^{(4)}{\bf P}_2^{(4)} & \dots & w_2^{(m)}{\bf P}_2^{(m)} \\ \\
& {\bf 0} & & w_3^{(4)}{\bf P}_3^{(4)} & \dots & w_3^{(m)}{\bf P}_3^{(m)} \\ \\
& & & & \ddots & \vdots \\ \\
& & & & & w_{m-1}^{(m)}{\bf P}_{m-1}^{(m)}
\end{array} \right) \,.\]  
In the sequel, we transform $A_0^{(m)}$ step by step by elementary matrix operations. 
After performing the $i$-th step $(i\geq 1)$ on $A_{i-1}^{(m)}$, we denote the
resulting $(m+3)\times (4m)$ - matrix by $A_i^{(m)}$. \\
{\em Step\,1.\/}\,Using elementary column operations, 
we transform ${\bf R}^{(2)}$ into a row echelon form, 
\[{\hat {\bf R}}^{(2)} \,:=\, \left( \begin{array}{cccc}
0 & -1/8 & 1/24 & -1/24 \\
0 & 1/8 & 0 & 0 \\
0 & 0 & 1/24 & 0 \\
0 & 0 & 0 & 1/24
\end{array} \right) \,.\]
{\em Step\,j\/}\,(for $2\leq j\leq m$). By elementary column operations, 
Lemma\,\ref{Lem1} allows us to transform the $j$-th row of the matrix $A_1^{(m)}$ 
into the row
\[\big( \underbrace{0,\dots,0}_{4(j-1)},w_{j-1}^{(j)}
{\bf P}_{j-1}^{(j)},\underbrace{0,\dots,0}_{4(m-j)} \big) \,.\]
After the $m$-th step, we achieve the $(m+3)\times (4m)$ - matrix
\[A_m^{(m)} \,:=\, \left( \begin{array}{cccccc}
{\hat {\bf R}}^{(2)} & * & * & * & \dots & * \\ \\
& w_1^{(2)}{\bf P}_1^{(2)} & 0 & 0 & \dots & 0 \\ \\
& & w_2^{(3)}{\bf P}_2^{(3)} & 0 & \dots & 0 \\ \\
& {\bf 0} & & w_3^{(4)}{\bf P}_3^{(4)} & \dots & 0 \\ \\
& & & & \ddots & \vdots \\ \\
& & & & & w_{m-1}^{(m)}{\bf P}_{m-1}^{(m)}
\end{array} \right) \, \] 
where each asterisk in the first row of $A_m^{(m)}$ represents (possibly different) 
a rational number. 
Thanks to Lemma\,\ref{lem: subspace of linear dependencies of auxiliary polynomials},
for each $j=1,\dots,m-1$, the three entries 
$\Theta_j^+,-\Lambda_j^-,\Lambda_j^+$ in ${\bf P}_{j}^{(j+1)}$ are linearly independent over
${\Q}$, whereas the four entries $-\Theta_j^-,\Theta_j^+,-\Lambda_j^-,\Lambda_j^+$ 
are linearly dependent over ${\Q}$. Moreover, we know that 
$\mbox{rank}_{\Q}\big( {\hat {\bf R}}^{(2)} \big) =3$. This implies that
\[\mbox{rank}_{\Q}\big( A_m^{(m)} \big) \,\geq \,3m \,.\]
Due to the rank theorem we obtain
\begin{equation}
\dim_{\mathbb{Q}} \big( \,\mbox{kernel}(A_m^{(m)})\,\big) \,=\, 4m 
- \mbox{rank}_{\Q}\big( A_m^{(m)} \big) \,\leq \,m \qquad (m\geq 1)\,.
\label{kernel1}
\end{equation}
In order to prove the inverse inequality of (\ref{kernel1}) 
we apply induction with respect to $m\geq 6$. Moreover, 
we show at the same time that every vector
$\underline{v}\in \mbox{kernel}(A_0^{(m)})$ satisfies 
(\ref{eq: additional assertion in Stein Conjecture}). 
Note that in the case $m=1$ we obviously have
\[ \dim_{\mathbb{Q}} \big( \,\mbox{kernel}(A_0^{(1)})\,\big) 
\,=\, \dim_{\mathbb{Q}} \big( \,\mbox{kernel}({\bf R}^{(2)})\,\big) 
\,=\, 4 - 3 \,=\, 1\,,\]
and that (\ref{eq: additional assertion in Stein Conjecture}) holds. 
For $m=2$ and $m=3$, we refer the reader to the example in the appendix, where the truth of
(\ref{eq: additional assertion in Stein Conjecture}) 
can be deduced from the explicit solutions given by the formulae
\[(-2u+v)\Phi_2 + u\Phi_2^* + (u-v)\Psi_2^* - 7v\Phi_4 + 8v\Phi_4^* 
+ v\Psi_4 \,=\, 0 \qquad (u,v\in {\Q}) \]
and
\begin{eqnarray*}
&& (-u-w)\Phi_2 + u\Phi_2^* + w\Psi_2^* + (u-v-w)\Phi_4 
+ v\Phi_4^* + (u-w)\Psi_4 \\ \\
&+& \Big( \,\frac{128}{3}u - \frac{16}{3}v 
- \frac{128}{3}w\,\Big) \Phi_6 + \Big( \,-\frac{124}{3}u 
+ \frac{31}{6}v + \frac{124}{3}w\,\Big) \Phi_6^* \\ \\
&+& \Big( \,-\frac{4}{3}u + \frac{1}{6}v 
+ \frac{4}{3}w\,\Big) \Psi_6^* \,=\, 0 \qquad (u,v,w\in {\Q}) \,,
\end{eqnarray*}
cf. \cite[Eq.\,(5.16), (5.17)]{Stein}. The desired inequality 
$\dim_{\mathbb{Q}} \big( \,\mbox{kernel}(A_m^{(m)})\,\big) \geq m$ 
and the specific properties from
(\ref{eq: additional assertion in Stein Conjecture}) 
for $m=4,5,6$ can be verified by computer-assisted computations. \\
Now, let the assertions be true for $m\geq 6$. 
Then there exist $m$ linearly independent (column) vectors from ${\Q}^{4m+4}$, 
in which the last four entries 
vanish. We proceed to construct a vector $\underline{u}$ 
from the kernel of $A_0^{(m+1)}$, whose last four components 
do not vanish simultaneously. To this end, we find 
the components of $\underline{u}$ from the bottom up in groups of four entries. 
In the sequel, we assume that $m$ is even; for odd $m$ we apply analogous arguments. 
By Lemma\,\ref{lem: subspace of linear dependencies of auxiliary polynomials}, 
the vector \[\tau_1\underline{v}_{m} \,:=\, \tau_1 
{\big( 1-2^{2m+1},\,2^{2m+1},\,1,\,0 \big)}^T\,,\]
where $\tau_1$ is a parameter to be fixed in due time, 
can be applied to form the last four entries of $\underline{u}$. 
This is the first step of the recursive 
construction of $\underline{u}$. 
To proceed with step $i+1$ (after step $i$), 
we rewrite (\ref{eq: hauptteil}) by grouping together terms containing 
$\Theta_{m-i-1}^{\pm},\Lambda_{m-i-1}^{\pm}$ and bracketing out. 
We obtain the equation 
\[\omega^{(\Theta_{m-i-1}^+)}(\tau_1,\dots,\tau_i,\eta_1)\cdot \Theta_{m-i-1}^+ 
+ \dots + \omega^{(\Lambda_{m-i-1}^-)}(\tau_1,\dots,\tau_i,\eta_4)\cdot 
\Lambda_{m-i-1}^- \,=\, 0\,.\]
The quantities 
\begin{eqnarray*}
&& \omega^{(\Theta_{m-i-1}^+)}(\tau_1,\dots,\tau_i,\eta_1) \,, \\
&& \omega^{(\Theta_{m-i-1}^-)}(\tau_1,\dots,\tau_i,\eta_2) \,, \\
&& \omega^{(\Lambda_{m-i-1}^+)}(\tau_1,\dots,\tau_i,\eta_3) \,, \\
&& \omega^{(\Lambda_{m-i-1}^-)}(\tau_1,\dots,\tau_i,\eta_4)
\end{eqnarray*}
are linear forms in $\tau_{1},\ldots,\tau_{i}$ with rational coefficients. 
Next, given a parameter $\tau_{i+1}\in {\Q}$, we can find numbers $\eta_1,\dots,\eta_4$ 
such that for every $i=1,\dots,m-1$ the vector identity
\[{\big( \omega^{(\Theta_{m-i-1}^+)}(\tau_1,\dots,\tau_i,\eta_1),\dots 
,\omega^{(\Lambda_{m-i-1}^-)}(\tau_1,\dots,\tau_i,\eta_4) \big)}^T 
= \tau_{i+1} {\big( 2^{2(m-i-1)+1},\,2^{2(m-i-1)+1}-1,0,1 \big)}^T \]
holds. Now we fix the parameters $\tau_1,\dots,\tau_m$. 
Let $\underline{u}=(u_1,\dots,u_{4m+4})$. If the equation
\begin{equation}
\sum_{i=1}^m {\bf R}^{(2i)}\cdot {\big( \,u_{4i+1},\dots,u_{4i+4}\,\big)}^T 
\,=\, \underline{0}
\label{EQSummeSkalarprodukt}
\end{equation}  
holds, then we let $(u_1,\dots,u_4)\in \mbox{kernel}\big( {\bf R}^{(2)} \big)$. 
Otherwise, if (\ref{EQSummeSkalarprodukt}) does not hold, we can choose 
$\tau_1,\dots,\tau_m$ in such a way that every non-trivial linear form 
with rational coefficients in the components on the left-hand side of 
(\ref{EQSummeSkalarprodukt}) vanishes. In order to achieve $\tau_1 \not= 0$, 
we exploit the condition $m\geq 6$, since we need two parameters to cancel
the first component, and one parameter for cancelling each 
of the three remaining components. This completes the proof by induction and shows that
\[\dim_{\mathbb{Q}} \big( \,\mbox{kernel}(A_m^{(m)})\,\big) \,\geq \,m \,.\]
Together with (\ref{kernel1}) we have proven that 
$\dim_{\mathbb{Q}} \big( \,\mbox{kernel}(A_m^{(m)})\,\big) =m$ 
for $m\geq 1$. Additionally, we know that 
(\ref{eq: additional assertion in Stein Conjecture}) holds 
by the construction of $\underline{u}$.
\end{proof}

\newpage

\section{Appendix: The matrix $A_0^{(3)}$} \label{Sec7}
Firstly, we compute step by step all the expressions of the 
12 functions $\Phi_{2j},\Phi_{2j}^*,\Psi_{2j},\Psi_{2j}^*$ 
for $j=1,2,3$ using the formulae (\ref{Phi.even} - \ref{Psi*.odd}). 
Secondly, we decompose all these expressions into the polynomials 
$\mathbf{P}_{I}^{\left(\xi\right)}$ and $\mathbf{P}_{II}^{\left(\xi\right)}$, 
such that the elements of the matrix $A_0^{(3)}$ are obtained.
\begin{center}
\begin{tabular}{l||cc|c|c} 
Tab.\,1 & $a_j$ & $b_j$ & $c_j(k)$ & $d_j(k)$ \\ \hline & & & \\
$j=0$ & $\displaystyle \frac{1}{3}$ & 1 & -- & -- \\ & & & \\
$j=1$ & $\displaystyle \frac{1}{15}$ & 1 
& $\displaystyle \frac{1}{15}(1-k^2+k^4)$ & $k^2(1-k^2)$ \\ & & & \\
$j=2$ & $\displaystyle \frac{2}{189}$ & $\displaystyle \frac{2}{3}$ 
& $\displaystyle \frac{1}{189}(1+k^2)(1-2k^2)(2-k^2)$ 
& $\displaystyle -\frac{1}{3}k^2(1-k^2)(1-2k^2)$ 
\end{tabular}
\end{center}
\[\]
\begin{center}
\begin{tabular}{l||c|c} 
Tab.\,2 & $\Theta_j^-(k)$ & $\Theta_j^+(k)$ \\ \hline & \\
$j=1$ & $\displaystyle \frac{1}{15}\big( 16k^4 - 16k^2 + 1 \big)$ 
& $\displaystyle \frac{1}{15}\big( 1 + 14k^2 - 14k^4 \big)$ \\ & \\
$j=2$ & $\displaystyle \frac{2}{189}(2k^2-1)(32k^4-32k^2-1)$ 
& $\displaystyle -\frac{2}{189}(2k^2-1)(31k^4-31k^2+1)$ 
\end{tabular}
\end{center}
\[\]
\begin{center}
\begin{tabular}{l||c|c|c|c} 
Tab.\,3 & $e_j(k)$ & $f_j(k)$ & $\Lambda_j^-(k)$ & $\Lambda_j^+(k)$ \\ \hline & & & \\
$j=1$ & $1-k^2$ & $-k^2$ & 1 & $1-2k^2$ \\ & & & \\
$j=2$ & $\displaystyle \frac{1}{3}(1-k^2)(2-k^2)$ 
& $\displaystyle \frac{1}{3}k^2(1+k^2)$ & $\displaystyle \frac{2}{3}(1-k^2)$ 
& $\displaystyle \frac{2}{3}(k^4-k^2+1)$
\end{tabular}
\end{center}
\[\]
\begin{center}
\begin{tabular}{l||c||c|c||c|c} 
& & $w_j^{(s)}$ & $w_j^{(s)}$ & $\sigma_{s-j-1}(s)$ & $\sigma_{s-j-1}(s)$ \\ 
Tab.\,4 & ${\hat w}_s$ & $j=1$ & $j=2$ & $j=1$ & $j=2$ \\ \hline & & & & \\
$s=1$ & $\displaystyle \frac{1}{24}$ & -- & -- & -- & -- \\ & & & & \\
$s=2$ & $\displaystyle \frac{1}{144}$ & $\displaystyle -\frac{1}{96}$ & -- & 1 & -- \\ & & & & \\
$s=3$ & $\displaystyle \frac{1}{720}$ & $\displaystyle \frac{1}{384}$ & $\displaystyle 
\frac{1}{640}$ & -5 & 1 
\end{tabular}
\end{center}
\[\]
The expressions for $x_1,\dots,x_{12}$ now follow 
from the formulae at the beginning of Section\,\ref{Sec4}, and from Tables\,2 and 3. 
Hence, we obtain with (\ref{Phi.even} -- \ref{Psi*.odd}) 
the following formulae.
\begin{eqnarray*}
\Phi_2 &=& \mathbf{P}_{I}^{(\Phi_2)} + \mathbf{P}_{II}^{(\Phi_2)} 
\,=\, \frac{1}{24}\left( \,1 - {\Big( \,\frac{2K}{\pi}\,\Big)}^2(1-2k^2)\,\right) \\
&=& \frac{1}{24}(x_1 + x_3) \,,\\
\Phi_2^* &=& \mathbf{P}_{I}^{(\Phi_2^*)} + \mathbf{P}_{II}^{(\Phi_2^*)} 
\,=\, \frac{1}{24}\left( \,{\Big( \,\frac{2K}{\pi}\,\Big)}^2\Big( \,\frac{6E}{K}-5+4k^2\,\Big) 
-1\,\right) \\
&=& \frac{1}{24}(x_4 - x_1) \,,\\
\Psi_2 &=& \mathbf{P}_{I}^{(\Psi_2)} + \mathbf{P}_{II}^{(\Psi_2)} 
\,=\, \frac{1}{8}\left( \,{\Big( \,\frac{2K}{\pi}\,\Big)}^2 - 1 \,\right) \\
&=& \frac{1}{8}(x_2 - x_1) \,,\\
\Psi_2^* &=& \mathbf{P}_{I}^{(\Psi_2^*)} + \mathbf{P}_{II}^{(\Psi_2^*)} 
\,=\, \frac{1}{8}\left( \,{\Big( \,\frac{2K}{\pi}\,\Big)}^2\Big( \,1 - \frac{2E}{K}\,\Big)
+ 1 \,\right) \\
&=& \frac{1}{24}(3x_1 + 2x_3 - x_4) \,,\\
\Phi_4 &=& \mathbf{P}_{I}^{(\Phi_4)} + \mathbf{P}_{II}^{(\Phi_4)} 
\,=\, \frac{1}{144}\left( \,{\Big( \,\frac{2K}{\pi}\,\Big)}^2\Big( \,\frac{6E}{K}-5+4k^2\,\Big) 
- 1 \,\right) \\
&& + \frac{1}{1440} \left( \,{\Big( \,\frac{2K}{\pi}\,\Big)}^4(16k^4 - 16k^2 +1) - 1\,\right) \\
&=& \frac{1}{144} (x_4 - x_1) + \frac{1}{96} \Big( \,x_5 - \frac{x_1}{15}\,\Big) 
\,=\, -\frac{11}{1440}x_1 + \frac{1}{144}x_4 + \frac{1}{96}x_5\,,\\
\Phi_4^* &=& \mathbf{P}_{I}^{(\Phi_4^*)} + \mathbf{P}_{II}^{(\Phi_4^*)} 
\,=\, \frac{1}{144}\left( \,1 - {\Big( \,\frac{2K}{\pi}\,\Big)}^2(1-2k^2)\,\right) \\
&& + \frac{1}{1440}\left( \,1 - {\Big( \,\frac{2K}{\pi}\,\Big)}^4(1 + 14k^2 - 14k^4)\,\right) \\
&=& \frac{1}{144} (x_1 + x_3) + \frac{1}{96} \Big( \,\frac{x_1}{15} - x_6\,\Big) 
\,=\, \frac{11}{1440}x_1 + \frac{1}{144}x_3 - \frac{1}{96}x_6 \,,\\
\Psi_4 &=& \mathbf{P}_{I}^{(\Psi_4)} + \mathbf{P}_{II}^{(\Psi_4)} 
\,=\, \frac{1}{48}\left( \,{\Big( \,\frac{2K}{\pi}\,\Big)}^2\Big( \,\frac{2E}{K} - 1\,\Big) 
- 1\,\right) \\
&& + \frac{1}{96} \left( \,{\Big( \,\frac{2K}{\pi}\,\Big)}^4 - 1\,\right) \\
&=& \frac{1}{144} (x_4 - 2x_3 - 3x_1) + \frac{1}{96} (x_7 - x_1) 
\,=\, -\frac{1}{32}x_1 - \frac{1}{72}x_3 + \frac{1}{144}x_4 + \frac{1}{96}x_7 \,,\\
\Psi_4^* &=& \mathbf{P}_{I}^{(\Psi_4^*)} + \mathbf{P}_{II}^{(\Psi_4^*)} 
\,=\, \frac{1}{48}\left( \,1 - {\Big( \,\frac{2K}{\pi}\,\Big)}^2\,\right) \\
&& + \frac{1}{96} \left( \,1 - {\Big( \,\frac{2K}{\pi}\,\Big)}^4 (1 - 2k^2)\,\right) \\
&=& \frac{1}{48} (x_1 - x_2) + \frac{1}{96} (x_1 - x_8) 
\,=\, \frac{1}{32}x_1 - \frac{1}{48}x_2 - \frac{1}{96}x_8 \,,\\
\end{eqnarray*}
\[\]
\[\]
\begin{eqnarray*}
\Phi_6 &=& \mathbf{P}_{I}^{(\Phi_6)} + \mathbf{P}_{II}^{(\Phi_6)} \\
&=& \frac{1}{720}\left( \,1 - {\Big( \,\frac{2K}{\pi}\,\Big)}^2(1-2k^2)\,\right) 
+ \frac{1}{5760} \left( \,1 - {\Big( \,\frac{2K}{\pi}\,\Big)}^4 (1 + 14k^2 - 14k^4)\,\right) \\
&& + \frac{1}{640}\left( \,\frac{2}{189} + 
{\Big( \,\frac{2K}{\pi}\,\Big)}^6\frac{2}{189}(2k^2-1)(31k^4-31k^2+1) \,\right) \\
&=& \frac{1}{720} (x_1 + x_3) + \frac{1}{384}\Big( \,\frac{1}{15}x_1 - x_6\,\Big) 
+ \frac{1}{640} \Big( \,\frac{2}{189}x_1 - x_9\,\Big) \\
&=& \frac{191}{120960}x_1 + \frac{1}{720}x_3 - \frac{1}{384}x_6 - \frac{1}{640}x_9\,,\\
\Phi_6^* &=& \mathbf{P}_{I}^{(\Phi_6^*)} + \mathbf{P}_{II}^{(\Phi_6^*)} \\
&=& \frac{1}{720}\left( \,
{\Big( \,\frac{2K}{\pi}\,\Big)}^2\Big( \,\frac{6E}{K}-5+4k^2\,\Big) 
- 1 \,\right) - \frac{1}{5760} \left( \,1 - 
{\Big( \,\frac{2K}{\pi}\,\Big)}^4 (16k^4 - 16k^2 +1)\,\right) \\
&& - \frac{1}{640}\left( \,\frac{2}{189} - 
{\Big( \,\frac{2K}{\pi}\,\Big)}^6\frac{2}{189}(2k^2-1)(32k^4-32k^2-1) \,\right) \\
&=& \frac{1}{720}(x_4 - x_1) - \frac{1}{384}\Big( \,\frac{1}{15}x_1 - x_5\,\Big) 
- \frac{1}{640} \Big( \,\frac{2}{189}x_1 - x_{10} \,\Big) \\
&=& -\frac{191}{120960}x_1 + \frac{1}{720}x_4 + \frac{1}{384}x_5 + \frac{1}{640}x_{10}\,,\\
\Psi_6 &=& \mathbf{P}_{I}^{(\Psi_6)} + \mathbf{P}_{II}^{(\Psi_6)} \\
&=& \frac{1}{240} \left( \,{\Big( \,\frac{2K}{\pi}\,\Big)}^2 - 1\,\right) 
+ \frac{1}{384} \left( \,{\Big( \,\frac{2K}{\pi}\,\Big)}^4 (1-2k^2) - 1\,\right) \\
&& + \frac{1}{960} \left( \,{\Big( \,\frac{2K}{\pi}\,\Big)}^6 (k^4-k^2+1) - 1\,\right) \\
&=& \frac{1}{240} (x_2 - x_1) + \frac{1}{384} (x_8 - x_1) + \frac{1}{640} 
\Big( \,x_{11} - \frac{2}{3}x_1\,\Big) \\ 
&=& -\frac{1}{128}x_1 + \frac{1}{240}x_2 + \frac{1}{384}x_8 + \frac{1}{640}x_{11}\,,\\ 
\Psi_6^* &=& \mathbf{P}_{I}^{(\Psi_6^*)} + \mathbf{P}_{II}^{(\Psi_6^*)} \\
&=& \frac{1}{240} \left( \,1 + {\Big( 
\,\frac{2K}{\pi}\,\Big)}^2\Big( \,1 - \frac{2E}{K}\,\Big) \,\right) 
+ \frac{1}{384} \left( \,1 - {\Big( \,\frac{2K}{\pi}\,\Big)}^4 \,\right) \\
&& + \frac{1}{960} \left( \,1 - {\Big( \,\frac{2K}{\pi}\,\Big)}^6 (1-2k^2)\,\right) \\
&=& \frac{1}{720} (3x_1 + 2x_3 - x_4) + \frac{1}{384} (x_1 - x_7) 
+ \frac{1}{640} \Big( \,\frac{2}{3}x_1 - x_{12}\,\Big) \\
&=& \frac{1}{128}x_1 + \frac{1}{360}x_3 - \frac{1}{720}x_4 
- \frac{1}{384}x_7 - \frac{1}{640}x_{12} \,.
\end{eqnarray*}
\[\]
We represent these 12 identities, which express each 
$\xi \in \Omega_2 \cup \Omega_4 \cup \Omega_6$ as linear forms 
in terms of $x_1,\dots,x_{12}$, by the matrix
\newpage
\[\left( \begin{array}{cccccccccccc}
\cfrac{x_1}{24} & \cfrac{-x_1}{24} & \cfrac{-x_1}{8} & \cfrac{x_1}{8} 
& \cfrac{-11x_1}{1440} & \cfrac{11x_1}{1440} & \cfrac{-x_1}{32} & \cfrac{x_1}{32} & 
\cfrac{191x_1}{120960} & \cfrac{-191x_1}{120960} & \cfrac{-x_1}{128} & \cfrac{x_1}{128} \\ \\
0 & 0 & \cfrac{x_2}{8} & 0 & 0 & 0 & 0 & \cfrac{-x_2}{48} & 0 & 0 & \cfrac{x_2}{240} & 0 \\ \\
\cfrac{x_3}{24} & 0 & 0 & \cfrac{x_3}{12} & 0 & \cfrac{x_3}{144} 
& \cfrac{-x_3}{72} & 0 & \cfrac{x_3}{720} & 0 & 0 & \cfrac{x_3}{360} \\ \\
0 & \cfrac{x_4}{24} & 0 & \cfrac{-x_4}{24} & \cfrac{x_4}{144} & 0 
& \cfrac{x_4}{144} & 0 & 0 & \cfrac{x_4}{720} & 0 & \cfrac{-x_4}{720} \\ \\
0 & 0 & 0 & 0 & \cfrac{x_5}{96} & \cfrac{-x_6}{96} & \cfrac{x_7}{96} 
& \cfrac{-x_8}{96} & \cfrac{-x_6}{384} & \cfrac{x_5}{384} & \cfrac{x_8}{384} & 
\cfrac{-x_7}{384} \\ \\
0 & 0 & 0 & 0 & 0 & 0 & 0 & 0 & \cfrac{-x_9}{640} & \cfrac{x_{10}}{640} 
& \cfrac{x_{11}}{640} & \cfrac{-x_{12}}{640} 
\end{array} \right) \,.\]
The 12 columns correspond to the 12 $\xi$-functions $\Phi_2,\dots,\Psi_6^*$. 
Row\,1 to row\,4 correspond to the linearly independent parameters $x_1,\dots, x_4$,
whereas row\,5 and row\,6 represent the sets of parameters $\{ x_5,x_6,x_7,x_8 \}$ 
and $\{ x_9,x_{10},x_{11},x_{12} \}$, respectively. Here, every parameter from
one set is linearly independent over ${\Q}$ from the parameters of the other set. 
This follows from the fact that the parameters $x_5,\dots,x_8$ are provided with 
the factor ${(2K/\pi)}^4$, the parameters $x_9,\dots,x_{12}$ with the factor 
${(2K/\pi)}^6$. Therefore, with regard to rank considerations, we simplify the above
matrix by replacing each $x_1,\dots,x_4$ in row\,1 to row\,4 by 1, 
and by removing the factors ${(2K/\pi)}^4$ and ${(2K/\pi)}^6$ from $x_5,\dots,x_8$ and 
$x_9,\dots,x_{10}$, respectively. Using Tab.\,2 to Tab.\,4 as well as the vectors
\begin{eqnarray*}
{\bf P}_1^{(2)} &=& \big( -\Theta_1^-,\Theta_1^+,-\Lambda_1^-,\Lambda_1^+\big) \,,\\
{\bf P}_1^{(3)} &=& \big( -\Theta_1^+,\Theta_1^-,\Lambda_1^+,-\Lambda_1^-\big) \,,\\
{\bf P}_2^{(3)} &=& \big( -\Theta_2^+,\Theta_2^-,\Lambda_2^+,-\Lambda_2^-\big) \,,\\
\end{eqnarray*}
we obtain the matrix $A_{0}^{(3)}$ given by
\[\left( \begin{array}{cccccccccccc}
\cfrac{1}{24} & \cfrac{-1}{24} & \cfrac{-1}{8} & \cfrac{1}{8} 
& \cfrac{-11}{1440} & \cfrac{11}{1440} & \cfrac{-1}{32} & \cfrac{1}{32} & 
\cfrac{191}{120960} & \cfrac{-191}{120960} & \cfrac{-1}{128} & \cfrac{1}{128} \\ \\
0 & 0 & \cfrac{1}{8} & 0 & 0 & 0 & 0 & \cfrac{-1}{48} & 0 & 0 & \cfrac{1}{240} & 0 \\ \\
\cfrac{1}{24} & 0 & 0 & \cfrac{1}{12} & 0 & \cfrac{1}{144} 
& \cfrac{-1}{72} & 0 & \cfrac{1}{720} & 0 & 0 & \cfrac{1}{360} \\ \\
0 & \cfrac{1}{24} & 0 & \cfrac{-1}{24} & \cfrac{1}{144} & 0 
& \cfrac{1}{144} & 0 & 0 & \cfrac{1}{720} & 0 & \cfrac{-1}{720} \\ \\
0 & 0 & 0 & 0 & -\cfrac{-\Theta_1^-}{96} & -\cfrac{\Theta_1^+}{96} 
& -\cfrac{-\Lambda_1^-}{96} & -\cfrac{\Lambda_1^+}{96} & 
\cfrac{-\Theta_1^+}{384} & \cfrac{\Theta_1^-}{384} 
& \cfrac{\Lambda_1^+}{384} & \cfrac{-\Lambda_1^-}{384} \\ \\
0 & 0 & 0 & 0 & 0 & 0 & 0 & 0 & \cfrac{-\Theta_2^+}{640} 
& \cfrac{\Theta_2^-}{640} & \cfrac{\Lambda_2^+}{640} & \cfrac{-\Lambda_2^-}{640} 
\end{array} \right) \,,\]
which (with a simpler notation) corresponds to 
\[A_0^{(3)} \,=\, \left( \begin{array}{ccc}
{\bf R}^{(2)} & {\bf R}^{(4)} & {\bf R}^{(6)} \\ 
{\bf 0} & w_1^{(2)}{\bf P}_1^{(2)} &  w_1^{(3)}{\bf P}_1^{(3)} \\
{\bf 0} & {\bf 0} & w_2^{(3)}{\bf P}_2^{(3)}
\end{array} \right) \,.\]
\[\]

{\bf{}Acknowledgements.}
The second author was partially supported by the Austrian 
Science Fund (FWF) projects W1230, Y-901, and EPSRC grant EP/J018260/1.

\newpage



\begin{thebibliography}{99}
\bibitem{Byrd} P.\,F.\,Byrd and M.\,D.\,Friedman, 
{\em Handbook of elliptic integrals for engineers and physicists\/}, 
second edition, Springer (1971). 
\bibitem{Jeannin} R.\,Andr{\'e}-Jeannin, 
{\em Irrationalit{\'e} de la somme des inverses 
de certaines suites r{\'e}currentes\/}, C.R.Acad.Sci. Paris S{\'e}r. I Math.
{\bf 308} (1989), 539 -- 541.
\bibitem{Duverney} D.\,Duverney, Ke.\,Nishioka, Ku.\,Nishioka,  and I.\,Shiokawa, 
{\em Transcendence of Rogers-Ramanujan continued fraction and reciprocal sums of
Fibonacci numbers\/}, Proc. Japan Acad., Ser.\,A, Math.Sci. {\bf 73} (1997), 140 -- 142.
\bibitem{Elsner1} C.\,Elsner, S.\,Shimomura, and I.\,Shiokawa, 
{\em Algebraic relations for reciprocal sums of binary recurrences\/}, 
Seminar on Mathematical Sciences:
Diophantine Analysis and Related Fields, Keio University, Yokohama, 
{\bf 35} (2006), 77 -- 92. 
\bibitem{Elsner2} C.\,Elsner, S.\,Shimomura, and I.\,Shiokawa, 
{\em Algebraic relations for reciprocal sums of Fibonacci numbers\/}, Acta Arithm. {\bf 130.1}
(2007), 37 -- 60.
\bibitem{Elsner3} C.\,Elsner, S.\,Shimomura, and I.\,Shiokawa, 
{\em Algebraic relations for reciprocal sums of odd 
terms in Fibonacci numbers\/}, Ramanujan J. {\bf 17}
(2008), 429 -- 446.
\bibitem{Elsner4} C.\,Elsner, S.\,Shimomura, and I.\,Shiokawa, 
{\em Exceptional algebraic relations for reciprocal sums of Fibonacci and Lucas numbers\/},
Diophantine Analysis and Related Fields 2011, AIP Conf. Proc., 
{\bf 1385}, M.Amou and M.Katsurada (Eds.), A.I.P., (2011), 17 -- 31.
\bibitem{Elsner5} C.\,Elsner, S.\,Shimomura, and I.\,Shiokawa, 
{\em Algebraic independence results for reciprocal sums of Fibonacci numbers\/}, 
Acta Arithm. {\bf 148.3} (2011), 205 -- 223.
\bibitem{Yuri} Yu.\,V.\,Nesterenko, {\em Modular functions 
and transcendence questions\/}, Mat.\,Sb. {\bf 187} (1996), 65 -- 96; 
English transl. Sb. Math. {\bf 187}
(1996), 1319 -- 1348. 
\bibitem{Stein} M.\,Stein, 
{\em Algebraic independence results for reciprocal sums 
of Fibonacci and Lucas numbers\/}, Ph.D. thesis of the Gottfried Wilhelm Leibniz
Universit{\"a}t Hannover, Fakult{\"a}t f{\"u}r Mathematik und Physik, (2012).\\
https://www.tib.eu/de/suchen/id/TIBKAT\%3A684662426/Algebraic-independence-results-for-reciprocal-sums/
\bibitem{Zucker} I.\,J.\,Zucker, {\em The summation of series of 
hyperbolic functions\/}, SIAM J.Math.Anal. {\bf 10} (1979), 192 -- 206.
\end{thebibliography}
\end{document}